\documentclass[a4paper,11pt]{amsart}
\usepackage{amsmath,amssymb,epsfig,rotating}
\usepackage{amsfonts}

\usepackage[only,mapsfrom,Mapsto,Mapsfrom]{stmaryrd} 

\evensidemargin \oddsidemargin \marginparwidth 0.5in 
\usepackage[all]{xy}
\usepackage[
  pdftitle={On prolongations of contact manifolds},
  pdfauthor={Mirko Klukas and Bijan Sahamie},
  citecolor=blue,
  colorlinks,
  linkcolor=blue,
  urlcolor=red,
  ]{hyperref}

\title{On prolongations of contact manifolds}

\author{Mirko Klukas}
\address{Mathematisches Institut der Universit\"at zu K\"oln,
Weyertal 86-90, 50931 K\"oln, Germany}
\email{mklukas@math.uni-koeln.de}
\urladdr{http://www.mi.uni-koeln.de/~mklukas}
\author{Bijan Sahamie}
\address{Mathematisches Institut der LMU M\"unchen, 
Theresienstrasse 39, 80333 M\"unchen, Germany}
\email{sahamie@math.lmu.de}
\urladdr{http://www.math.lmu.de/~sahamie}

\theoremstyle{plain} 
\newtheorem{theorem}{Theorem}[section]   
\newtheorem{lem}[theorem]{Lemma}         
\newtheorem{prop}[theorem]{Proposition}
\newtheorem{cor}[theorem]{Corollary}

\theoremstyle{definition}

\newtheorem*{ackn}{Acknowledgements}
            
\numberwithin{equation}{section}

\newcommand{\e}{\ensuremath{e}}
\newcommand{\N}{\mathbb{N}}
\newcommand{\pxi}{\mathbb{P}(\xi)}

\newcommand{\mD}{\mathcal{D}}
\newcommand{\sone}{\mathbb{S}^1}
\newcommand{\Z}{\mathbb{Z}}
\newcommand{\lra}{\longrightarrow}
\newcommand{\co}{\colon\thinspace}

\newcommand{\stwo}{\mathbb{S}^2}
\newcommand{\R}{\mathbb{R}}

\newcommand{\Hom}{\mbox{\rm Hom}}
\newcommand{\pfat}{{\boldsymbol p}}
\newcommand{\afat}{{\boldsymbol \alpha}}

\DeclareMathSizes{11}{11}{8}{7}

\begin{document}
%
%
\fontsize{11}{14}\selectfont
\begin{abstract}
We apply spectral sequences to derive both
an obstruction to the existence of $n$-fold prolongations
and a topological classification. Prolongations
have been used in the literature in an attempt to prove that every Engel
structure on $M\times\sone$ with characteristic
line field tangent to the fibers is determined by the contact
structure induced on a cross section and the twisting of 
the Engel structure along the fibers. Our results show
that this statement needs some modification: to classify the
diffeomorphism type of the Engel structure we additionally have to
fix a class in the first cohomology of $M$.
\end{abstract}
\maketitle

%
\section{Introduction}\label{sec:i}
%
%
\noindent The goal of this note is to present a discussion 
of prolongations of contact manifolds. For an introduction to basic 
notions of Engel structures, we point 
the reader to \cite[\S 2.2]{monty}, \cite[\S 1.2]{adachi} 
and \cite{vogel}. Given 
a $3$-dimensional contact manifold $(M,\xi)$, we define its 
prolongation $\pxi$ as the $\sone$-bundle over $M$ obtained by 
projectivizing the contact planes $\xi$ 
(cf.~\S\ref{sec:adachicorrected} or see~\cite[\S 2.2]{monty}
and \cite[\S 1.2]{adachi}). In \cite{adachi}, 
Adachi discusses prolongations of contact manifolds and 
introduces a notion of an {\it $n$-fold prolongation}, which 
he defines as a \textit{fiberwise $n$-fold covering} of $\pxi$. 
This notion is then employed in an attempt to prove that Engel 
structures $\mD$ 
on $4$-manifolds $M\times\sone$, where $M$ is a closed, oriented 
$3$-manifold and $\mD$ is an Engel structure with characteristic 
foliation tangent to the $\sone$-fibers, are determined by the contact
structure induced on a cross section of $M\times\sone$ and the {\it twisting}
of $\mD$ along the $\sone$-fibers (see~\cite[Theorem~1(2)]{adachi}). His 
definition of the $n$-fold 
prolongation and the statement in his Theorem~1(2) 
in \cite{adachi} suggest that $n$-fold prolongations of contact 
manifolds always exist and are unique. Additionally, the proof of 
Theorem~1(2) in \cite{adachi} seems to rest in an essential way on a lifting 
argument, which however does not work in general 
(cf.~\S\ref{sec:aie}). In \S\ref{sec:aie} we discuss an 
example which shows that the data specified by Adachi 
in Theorem 1(2) in \cite{adachi} are not sufficient to determine the
isotopy class of the Engel structure $\mD$ on $M\times\sone$. 
In \S\ref{sec:pocm} we apply methods from 
spin geometry to provide a topological classification of $n$-fold 
prolongations.
\begin{theorem}\label{corol} A contact manifold $(M,\xi)$ admits an $n$-fold 
prolongation if and only if the mod-$n$ reduction $e_n(\pxi)$ of the 
Euler class $e(\pxi)$ vanishes. The isomorphism classes of $n$-fold 
prolongations are in one-to-one correspondence with elements 
of $H^1(M;\Z_n)$.
\end{theorem}
\noindent The main ingredient in the proof of this theorem is the 
classification of fiberwise $n$-fold coverings of $\sone$-bundles over 
a given manifold $M$. We provide this classification in Theorem~\ref{thm:main}. 
Its proof uses techniques coming from the characterization of spin structures. 
As a corollary, we are able to show that Engel structures on an
$\sone$-bundle over $M$ with
characteristic foliation tangent to the fibers are classified by the 
set of data specified by Adachi and a class in $H^1(M;\Z_n)$ (see Corollary~\ref{thm:main B}).
\vspace{0.3cm}\\
Finally, in \S\ref{sec:oecop} we investigate which contact structures 
admit $n$-fold prolongations. We prove the following general 
existence result. 
\begin{theorem}\label{prop:main2} Every $3$-dimensional contact 
manifold $(M,\xi)$ admits both $2$-fold prolongations 
and $4$-fold prolongations.
\end{theorem}
\noindent This is derived by showing that the Euler class 
$e(\pxi)$ lies in the subgroup $4\cdot H^2(M;\Z)$ of $H^2(M;\Z)$.
Furthermore,  we see that, in fact, each element $q\in 4\cdot H^2(M;\Z)$ 
is given as the Euler class $e(\pxi)$ of the prolongation of a suitable 
contact structure $\xi$. Consequently, for every $n\not=2,4$ there 
is a contact structure which does not admit an $n$-fold 
prolongation.
\begin{ackn} The first author wishes to thank Hansj\"org Geiges for 
pointing his interest to the subject and for useful conversations.
We wish to thank the referee, the editor and especially Hansj\"org Geiges 
for useful comments which helped to improve the exposition.
\end{ackn}
%
%
\section{An Introductory Example}\label{sec:aie}
%
%
\noindent In this section we provide an example which 
illustrates that Theorem~1(2) 
in \cite{adachi} needs some modification. We will investigate
fiberwise $n$-fold coverings of the $\sone$-bundle $T^3\times\sone\lra T^3$.
The considerations presented in this section show that fiberwise
$n$-fold coverings are in one-to-one correspondence with elements
of $H^1(T^3;\Z_n)$ (cf.~Theorem~\ref{thm:main}).\vspace{0.3cm}\\
We consider the $3$-dimensional torus $T^3$ with coordinates $(x,y,z)$. Given 
some vector $\afat\in\Z^3\cap[0,n-1]^3$, we define a fiberwise $n$-fold covering
$\phi_\afat\co T^3\times\sone \to T^3 \times \sone$ 
by 
\[
  \phi_\afat(\pfat,\theta)
  =
  \big(\pfat , n\theta+\left<\afat,\pfat\right> \big).
\]
Here we use the identification $\sone=\R/\Z$. The restriction
of $\phi_\afat$ to a fiber of $T^3\times\sone$ corresponds to the 
unique connected $n$-fold covering of the circle. Observe 
that $\afat$ determines a cohomology class in $H^1(T^3;\Z_n)$: by the 
universal coefficient theorem we know that the group $H^1(T^3;\Z_n)$ is 
isomorphic to the group $\Hom(H_1(T^3;\Z);\Z_n)$, whose elements --~by fixing 
the {\it standard} generators of $H_1(T^3;\Z)$~-- are uniquely determined 
by a vector in $\Z^3\cap[0,n-1]^3$.\\
For $\afat$ and $\afat'$ with $\afat\not=\afat'$, the coverings $\phi_\afat$ and $\phi_{\afat'}$ are not equivalent: Suppose 
they were equivalent, then there exists a covering isomorphism $\psi$ 
such that $\phi_\afat\circ \psi=\phi_{\afat'}$. The morphism $\psi_*$ 
on fundamental groups is given by the matrix
\[
 \psi_*
 =
 \left(
 \begin{matrix}
 1 & 0 & 0 & 0 \\
 0 & 1 & 0 & 0 \\
 0 & 0 & 1 & 0 \\
 a & b & c & 1
 \end{matrix}
 \right)
\]
for suitable integers $a,b,c\in\Z$ and we have that 
$(\phi_\afat)_*\psi_*=(\phi_{\afat'})_*$. This implies the equality
$
  \afat'
  =
  \afat
  +n\cdot (a,b,c)
$, which means that $\afat'=\afat$ considering the fact that they are both
vectors in $\Z^3\cap[0,n-1]^3$.\vspace{0.3cm}\\
Conversely, given a fiberwise $n$-fold covering 
$\phi\co T^3\times\sone \to T^3\times\sone$ we associate with it a cohomology 
class $\afat$ in $H^1(T^3;\Z_n)$ as follows:
As outlined above, we have to assign an element in $\Z_n$ to each of
the standard generators of $H_1(T^3;\Z)$. Take such a 
standard generator, $c$ say, and lift it to an
embedded circle $\widetilde{c}\subset T^3 \times \sone$.
Write $p_2$ for the projection of $T^3 \times \sone$ to the
$\sone$-factor. Then we define
\[
  \afat(c)=\mbox{\rm deg}(p_2\circ\phi\circ\widetilde{c})\;\mbox{\rm mod } n.
\]
Observe that this is well defined, since $\phi$ is a fiberwise $n$-fold
covering. This algorithm defines a morphism $\afat\co H_1(T^3;\Z)\lra\Z_n$,
which corresponds to a class in $H^1(T^3;\Z_n)$ by the universal coefficient 
theorem.
%
%
\subsection{An Example of Non-Equivalent Prolongations}
%
%
Choose $\xi$ to be the standard contact structure 
on $T^3$, i.e.~$\xi$ is defined as the kernel of the contact 1-form
given by
$
 \sin(2\pi z)dx+\cos(2\pi z)dy.
$
The contact planes are spanned by $\partial_z$ and
\[
	V_\pfat
	=\cos(2\pi z)\thinspace\partial_x+\sin(2\pi z)
	\thinspace\partial_y
\]
where $\pfat=(x,y,z)$. The prolongation $\mathbb{P}(\xi)$ can be 
naturally identified with the $4$-dimensional torus 
$T^3\times\sone$ with coordinates $(x,y,z,\theta)$ and the 
corresponding Engel structure $\mD(\xi)$ is 
spanned by the tangent vectors $\partial_\theta$ and
$
  \cos(\pi\theta)\thinspace\partial_z
  + 
  \sin(\pi\theta)\thinspace V_\pfat.
$
Using the fiberwise $n$-fold covering $\phi_\afat$ from above, we can
pull back the Engel structure $\mD(\xi)$, so that we get a new Engel 
structure, which we denote by $\mD^n_\afat(\xi)$. At the 
point $(\pfat,\theta)$ the Engel plane
$\mD^n_\afat(\xi)_{(\pfat,\theta)}$ is spanned by the tangent vectors
$\partial_\theta$ and
\[
 \cos\big(\pi\big(n\theta + \left<\afat,\pfat\right>\big)
 \big)
 \thinspace
 \partial_z
 +
 \sin\big(\pi\big(n\theta + \left<\afat,\pfat\right>\big)
 \big)
 \thinspace
 V_\pfat.
\]
The Engel manifolds $(T^3\times\sone,\mD^n_\afat(\xi))$ are pairwise 
non-equivalent if we allow only isotopies through Engel structures 
with characteristic
foliation tangent to the $\sone$-fibers. These {\it special} isotopies 
appear in the proof of Corollary~\ref{thm:main B}. If one considers
general isotopies through Engel structures, then it is not clear which of 
these structures are isotopic.
%
%
\section{Characterization of Fiberwise Coverings}\label{sec:pocm}
%
%
\noindent Let us denote by $\varphi_n\co\sone\to\sone$ the connected 
$n$-fold covering of the unit circle $\sone$, where $n$ is some 
positive integer. 
Suppose we are given an $\sone$-bundle $P\to M$ over a closed, oriented 
manifold $M$. We define a \textbf{fiberwise $n$-fold covering of $P$} 
as a pair $(Q,\phi)$, where $Q$ is an $\sone$-bundle over $M$ and $\phi$ 
is a smooth map $Q \to P$ such that its restriction $\left.\phi\right|_{Q_x}$, for every $x\in M$, is a map $Q_x\to P_x$ which corresponds to the $n$-fold
covering map $\varphi_n\co\sone\to\sone$. Before we move our focus to the characterization of fiberwise $n$-fold coverings, we show that their 
existence is tied to the following condition on the Euler classes of the bundles. 
\begin{lem}\label{lem:help}
Let $\pi_Q\co Q\to M$ and $\pi_P\co P\to M$ be two principal $\sone$-bundles 
such that there is a bundle map $\phi\co Q\to P$ whose restriction 
to the fibers corresponds to the map $\varphi_n$, then 
$\e(P)=n\cdot \e(Q)$.
\end{lem}
\begin{proof} The Euler class provides an isomorphism 
\[
  H^1(M;\sone)\overset{e}{\lra} H^2(M;\Z),
\]
where $H^1(M;\sone)$ naturally corresponds to equivalence 
classes of principle $\sone$-bundles over $M$. So, the statement of the lemma can
be rephrased in terms of \v{C}ech cochains: Let 
$\mathfrak{U}=\{U_\alpha\}_{\alpha}$ be an open covering of $M$, 
let $\{g^Q_{\alpha\beta}\}$ be the \v{C}ech cochain representing 
the bundle $Q$ and let $\{g^P_{\alpha\beta}\}$ be the 
\v{C}ech cochain representing $P$. 
We have to prove that the \v{C}ech cohomology classes 
$[\{n\cdot(g^Q_{\alpha\beta})\}]=n\cdot[\{(g^Q_{\alpha\beta})\}]$ 
and $[\{g^P_{\alpha\beta}\}]$ are equal.\vspace{0.3cm}\\
Associated to the open covering $\mathfrak{U}$, the bundles $Q$ and $P$
admit bundle charts $\{t^Q_\alpha\}$ and $\{t^P_\alpha\}$. Now consider
the maps $\phi_\alpha=t^P_{\alpha}\circ\phi\circ (t^Q_\alpha)^{-1}$. Since
the restriction of $\phi$ to the fibers of $Q$ corresponds to 
$\varphi_n$, the equality
\[
 \phi_\alpha(p,\theta)=(p,\mu_\alpha(p)+n\cdot\theta)
\]
holds for suitable $\mu_\alpha\co U_\alpha\lra\sone$. Because 
the $\phi_\alpha$ come from the well-defined bundle map $\phi$, we 
have that
\[
 \phi_\beta(g^Q_{\alpha\beta}(p,\theta))
 =
 g^P_{\alpha\beta}(\phi_\alpha(p,\theta))
\]
for all $p\in U_\alpha\cap U_\beta$ and 
$\theta\in\sone$. This is equivalent to saying that
\[
 \mu_\beta(p)
 +n\cdot g^Q_{\alpha\beta}(p)
 -\mu_\alpha(p)
 =g^P_{\alpha\beta}(p)
\]
for all $p\in U_\alpha\cap U_\beta$. Hence, 
$\{n\cdot(g^Q_{\alpha\beta})\}$ and $\{g^P_{\alpha\beta}\}$ 
differ by the boundary of the \v{C}ech-$1$-cochain 
$\{\mu_\alpha\}$ and the result 
follows (cf.~\cite[Appendix A]{LaMi}).
\end{proof}
\noindent From the relationship of the Euler classes presented in 
Lemma~\ref{lem:help} we see that the mod $n$ reduction of the Euler
class is an obstruction to the existence of fiberwise $n$-fold coverings, 
i.e.~the existence of fiberwise $n$-fold coverings imply the vanishing 
of the mod $n$ reduction of the Euler class. In fact, the converse statement
is true as well (see~Theorem~\ref{thm:main}). Its proof will occupy the 
remainder of this section.
As in the characterization of 
spin structures used in \cite[\S 1]{LaMi} 
(see~especially sequences~(1.2) and (1.4) of \cite{LaMi}), we 
find it opportune to work with \v{C}ech cohomology. 
\begin{prop}\label{prop:prep} An $\sone$-bundle 
$\sone\lra P\lra M$ induces the following long exact sequence
\[
 \xymatrix@C=2pc@R=0.1pc
 {
 0\ar[r]
 &
 H^1(M;\Z_n)\ar[r]^{\pi^*}
 & 
 H^1(P;\Z_n)
 \ar[r]^{\iota^*}
 &
 H^1(\sone;\Z_n)\ar[r]^{d}
 &
 H^2(M;\Z_n),
 }
\]
where the map $d$ sends the generator of $H^1(\sone;\Z_n)$ to
the $mod\;n$ reduction of the Euler class $\e(P)$.
\end{prop}
\noindent To give a bit of explanation, observe, that 
an $n$-fold covering of the space $P$ corresponds to an element 
in $H^1(P;\Z_n)$ (cf.~\cite[Appendix A]{LaMi}). A fiberwise $n$-fold 
covering $(Q,\phi)$ is an 
ordinary $n$-fold covering and, thus, we may think of the pair 
$[(Q,\phi)]$ (or simply $[Q]$) as an element in the first 
cohomology of $P$. The 
statement that it is fiberwise $\varphi_n$ is equivalent to saying 
that the pullback bundle $\iota^*(Q,\phi)$ is isomorphic 
to $\varphi_n$. In terms of the exact sequence presented 
in Proposition~\ref{prop:prep}, this amounts to the equality 
$\iota^*[Q]=[\varphi_n]$ where $[\varphi_n]$ is a 
generator of $H^1(\sone;\Z_n)$ (cf.~\cite[\S 1]{LaMi} 
and \cite[Appendix~A]{LaMi}).

\begin{theorem}\label{thm:main} Given an $\sone$-bundle $P \to M$, a
fiberwise $n$-fold covering 
of $M$ exists if and only if
the mod $n$ reduction $\e_n(E)$ of the Euler class $\e(E)$ is zero.
In this case, the isomorphism classes of fiberwise $n$-fold 
coverings
of $M$ are in one-to-one correspondence with elements of
$H^1(M;\Z_n)$.
\end{theorem}
\begin{proof}
By Proposition~\ref{prop:prep} the following sequence is exact:
\[
 \xymatrix@C=2pc@R=0.1pc
 {
 0\ar[r]
 &
 H^1(M;\Z_n)\ar[r]^{\pi^*}
 & 
 H^1(P;\Z_n)
 \ar[r]^{\iota^*}
 &
 H^1(\sone;\Z_n)\ar[r]^{d}
 &
 H^2(M;\Z_n).
 }
\]
The mod $n$ reduction of the Euler class 
$\e(P)$ will be denoted by $\e_n$. Suppose that a
fiberwise $n$-fold covering $Q$ of $P$ exists. Then $[Q]$ is an 
element of $H^1(P;\Z_n)$ such that 
$\iota^*[Q]=[\varphi_n]$. Thus, by exactness of the 
sequence,
\[
  0=d(\iota^*[Q])=d[\varphi_n]=e_n.
\]
Conversely, assuming that $e_n=0$, we have 
$d[\varphi_n]=e_n=0$. By exactness, this implies
the existence of an element $q\in H^1(P;\Z_n)$ which is 
mapped to $[\varphi_n]$ under $\iota^*$. But $q$ corresponds 
to a fiberwise $n$-fold covering of $P$.\\
The isomorphism classes of fiberwise coverings correspond 
to the set $(\iota^*)^{-1}([\varphi_n])$ on which 
$\pi^*$ --~by the sequence above~-- induces a free and 
transitive $H^1(M;\Z_n)$-action. Hence, we obtain a 
one-to-one correspondence between $H^1(M;\Z_n)$ and the 
isomorphism classes of fiberwise coverings. 
\end{proof}
\noindent It remains to prove Proposition~\ref{prop:prep}. We just 
sketch the proof, since it is analogous 
to the proofs of the exact sequences used in the 
characterization of spin structures (see~\cite[\S1]{LaMi}).
\begin{proof}[Sketch of Proof of Proposition~\ref{prop:prep}]
We look at the Leray-Serre spectral sequence with $E_2$-page given by
$E_2^{p,q}=H^p(M;H^q(\sone;\Z_n))$ (see~\cite{cleary}).
By applying the fact that $H^q(\sone;\Z_n)$ is non-zero 
for $q=0,1$ only, we see that $E_\infty^{1,0}=E_2^{1,0}=H^1(M;\Z_n)$ 
and that $E_\infty^{0,1}=E_3^{0,1}$. Thus, we obtain the following 
exact sequence
\[
 \xymatrix@C=2pc@R=0.1pc
 {
 0\ar[r]
 &
 H^1(M;\Z_n)\ar[r]
 & 
 H^1(P;\Z_n)
 \ar[r]
 &
 E_2^{0,1}\ar[r]^{d_2^{0,1}}
 &
 H^2(M;\Z_n).
 }
\]
In fact, it is not hard to see that $E_2^{0,1}$ equals $H^1(\sone;\Z_n)$ and 
we obtain the exact sequence
as proposed. With a discussion similar to the spin 
case (see~\cite{Klukas} and cf.~\cite[\S 1]{LaMi}) it is possible to prove 
that $d_2^{0,1}$ sends the generator of $H^1(\sone;\Z_n)$ to the 
mod-$n$ reduction of the Euler class $\e(P)$.
\end{proof}
%

%
\section{Engel Structures with trivial Characteristic Line Field}\label{sec:adachicorrected}
%
%
\noindent An \textbf{Engel structure} is a maximally non-integrable $2$-plane
distribution $\mD$ on a $4$-dimensional manifold $Q$, i.e.~$\mD$ is defined
as a $2$-plane bundle for which $\mathcal{E}=[\mD,\mD]$ is of rank $3$ and
$[\mathcal{E},\mathcal{E}]$ of rank $4$. Inside the Engel structure $\mD$ there is
a line field $\mathcal{L}$ given by the condition that 
$[\mathcal{L},\mathcal{E}]\subset\mathcal{E}$. This line field is 
called the {\bf characteristic line field} and its induced foliation 
on $Q$ the
{\bf characteristic foliation} of $\mD$.
Engel structures arise in a natural way as prolongations of contact
$3$-manifolds. Given a contact $3$-manifold $(M,\xi)$ one can
consider the bundle $\mathbb{P}\xi$ whose fibers are the projectivizations 
of the contact planes, i.e.~for every $p\in M$ a 
point $q\in(\mathbb{P}\xi)_p$ corresponds to a line $l\subset \xi_p$ in the contact
plane. Note that by construction this
$4$-manifold carries the structure of an $\sone$-bundle 
$\rho:\mathbb{P}\xi\to M$ over $M$. Furthermore, we obtain a natural 
plane distribution $\mathcal{D}\xi \subset T \thinspace \mathbb{P}\xi$ 
given by
\[
  (\mathcal{D}\xi)_q = T_q\rho^{-1}(l).
\]
This distribution defines an Engel structure whose
characteristic line field is tangent to the 
$\sone$-fibers of the bundle.\vspace{0.3cm}\\
Now, assume we are given an oriented $\sone$-bundle $\pi: Q \to M$ over some
$3$-manifold $M$ and an Engel structure $\mD \subset TQ$ with
characteristic line field $\mathcal{L}$ tangent to the fibers. Since the
induced distribution $[\mathcal{D},\mathcal{D}] \subset TQ$ is
preserved by any flow tangent to $\mathcal{L}$, we obtain a well-defined 
contact structure $\xi=\pi_*([\mathcal{D},\mathcal{D}])$ on $M$. 
Furthermore, one obtains the {\bf development map}
\[
  \phi_\mathcal{D}\colon\thinspace  (Q,\mathcal{D})
  \longrightarrow (\mathbb{P}\xi, \mathcal{D}\xi)
\]
by assigning to a point $q \in Q$ the element 
$\phi_\mD(q)$ in $(\pxi)_{\pi(q)}$ which
corresponds to the $1$-dimensional subspace $T_q\pi(\mathcal{D}_q)$
of the contact plane $\xi_{\pi(q)}$. Note that $\phi_\mD(\mD)=\mD\xi$ and 
that $\phi_\mD$ defines a fiberwise 
$n$-fold covering of $\pxi$, where $n\in \N$ denotes the degree of the 
development map restricted to a fiber. 
We will refer to such $(Q,\mD)$ as an \textbf{$n$-fold prolongation 
of $(M,\xi)$}.
\begin{proof}[Proof of Theorem~\ref{corol}] 
Suppose that $e_n(\pxi)$ vanishes. Then Theorem~\ref{thm:main} implies 
the existence of a fiberwise $n$-fold covering 
$\phi\co Q\lra \pxi$. The space $Q$ is itself a $\sone$-bundle 
over $M$. We define an Engel structure $\mD$ on $Q$ by setting
$\mathcal{D} = (T\phi)^{-1}\big(\mD(\xi)\big)$.\\
Conversely, given an $n$-fold prolongation, the development map is a fiberwise
$n$-fold covering of $\pxi$, which by Lemma~\ref{lem:help} or Theorem~\ref{thm:main}
implies the vanishing of $e_n(\pxi)$.\\
By the statement of Theorem~\ref{thm:main} and the discussion from above we see 
that --~in case of existence~-- $n$-fold prolongations are in one-to-one 
correspondence with elements of $H^1(M;\Z_n)$.
\end{proof}
%
%
%
\begin{cor}
\label{thm:main B}
Suppose we are given an oriented Engel structure $\mD$ on 
an $\sone$-bundle $Q$ over $M$ whose
characteristic line field is tangent to the $\sone$-fibers.  Denote by $\xi$
the contact structure on the base given by $\xi=\pi_*([\mD,\mD])$ where $\pi$ is
the projection of $Q$ onto $M$. Let $n \in \N$ denote the twisting
number of $\mD$ and let $\alpha$ be a suitable class in $H^1(M;\Z_n)$ associated
to $\mD$. Then $\mD$ is classified up to Engel diffeomorphism by the set of 
data $(\xi,n,\alpha)$.
\end{cor}
\begin{proof}
Let $Q$ be an $\sone$-bundle over a closed, oriented $3$-manifold $M$ and
let $\mD_0$ and $\mD_1$ be two Engel structures on $Q$ which induce the same 
set of data $(\xi,n,\alpha)$. According to our classification of 
fiberwise $n$-fold coverings of $\pxi$ the 
$n$-fold prolongations $(Q,\mD_0)$ and $(Q,\mD_1)$ are both 
equivalent, i.e.~there is an isomorphism $\psi$ which makes the 
following diagram commutative
\[
 \xymatrix@C=3pc@R=3pc
 {
  (Q,\mD_0) \ar[rr]^{\psi}
  \ar[dr]^{\phi_{\mD_0}} & &
  (Q,\mD_1) \ar[dl]_{\phi_{\mD_1}}\\
  & \mathbb{P}(\xi) &
 }
\]
Conversely, to every triple $(\xi,n,\alpha)$ we can construct an Engel structure
with these data.
\end{proof}

\section{On Euler Classes of Prolongations}\label{sec:oecop}
\noindent In \S\ref{sec:pocm} we derived a characterization of 
fiberwise $n$-fold coverings and we have seen that the mod $n$ reduction 
of the Euler class determines their 
existence (see~Theorem~\ref{thm:main}). In \S\ref{sec:adachicorrected} 
we have seen that every $n$-fold prolongation $(Q,\mD)$ 
of a contact manifold $(M,\xi)$ naturally carries the structure of a fiberwise 
$n$-fold covering of the prolongation $\pxi$ via the development 
map $\phi_\mD$ (see~Theorem~\ref{corol}).
Thus, the existence of $n$-fold prolongations of contact manifolds is connected
to the vanishing of the mod $n$ reduction of the Euler class $e(\pxi)$.
This section is devoted to determining which contact structures admit $n$-fold prolongations. We start with the proof of Theorem~\ref{prop:main2}.
\begin{proof}[Proof of Theorem~\ref{prop:main2}]
Suppose we are given an oriented contact manifold $(M,\xi)$.
Choose a Riemannian metric on $M$ and a trivialization of the 
tangent bundle $TM$. Then the 2-plane field $\xi$ can be described in terms
of its corresponding Gau\ss \ map $f_\xi:M \to S^2$, which assigns 
to each $x \in M$ the positive normal vector to $\xi_x$. In 
fact, $f_\xi^*TS^2=\xi$. Denote by $u_0$ the positive generator of
$H^2(\stwo;\Z)$. Combining naturality of 
the Euler class under pullback with the fact that $\e(TS^2) = 2\cdot u_0$, 
we conclude that 
$
  \e(\xi) = 2 \cdot f_\xi^*u_0.
$
The unit-sphere bundle $\xi_1$ of $\xi$ is a fiberwise $2$-fold
covering of the prolongation $\pxi$, so that Lemma~\ref{lem:help}
implies 
$
 \e(\pxi)
 =
 2\cdot
 \e(\xi_1).
$
Because of $\e(\xi)=\e(\xi_1)$ we have
\[
  \e(\pxi) = 4 \cdot f_\xi^*u_0.
\]
Hence, both the mod $4$ reduction $\e_4(\pxi)$ and the mod $2$ 
reduction $\e_2(\pxi)$ of the Euler class $\e(\pxi)$ vanish. 
By Theorem~\ref{thm:main} the result follows.
\end{proof}
%
\noindent As a consequence of the last proof we see that $n$-fold
prolongations of contact manifolds $(M,\xi)$ often do not exist: Recall
that the class $f_\xi^*u_0$ classifies the homotopy type of
$\xi$ over the $2$-skeleton of $M$ (cf.~\cite[\S4.2]{Geiges}).
Since in every homotopy class of $2$-plane fields there is an overtwisted
contact structure, it is easy to find examples for which 
$\e(\pxi) = 4 \cdot f_\xi^*u_0$ is non-zero when reduced modulo $n$.


\begin{thebibliography}{99}

\bibitem{adachi}
{\sc J.~Adachi}, 
Engel structures with trivial characteristic foliations,
{\it Algebr. Geom. Topol.} {\bf 2} (2002), 239--255.

\bibitem{Bott}
{\sc R. Bott and L. W. Tu}, 
{\it Differential Forms in Algebraic Topology}, 
Graduate Texts in Mathematics {\bf 82}, Springer Verlag, 1982.

\bibitem{Geiges}
{\sc H. Geiges}, 
{\it An Introduction to Contact Topology}, 
Cambridge Studies in Advanced Mathematics {\bf 109}, 
Cambridge University Press, 2008.

\bibitem{Klukas}
{\sc M. Klukas}, 
{\it Engelstrukturen}, 
Diploma thesis, Universit\"{a}t zu K\"{o}ln, (2008)

\bibitem{LaMi}
{\sc H. B. Lawson and M.~Michelsohn}, 
{\it Spin Geometry},
Princeton Mathematical Series {\bf 38}, 
Princeton University Press, 1994.

\bibitem{cleary}
{\sc J. McCleary}, 
{\it A User's Guide to Spectral Sequences}, 
Cambridge Studies in Advanced Mathematics {\bf 58}, 
Cambridge University Press, 2001.

\bibitem{monty}
{\sc R. Montgomery}, 
Engel deformations and contact structures, 
American Mathematical Society Translations {\bf 196}(2), 103--117, American 
Mathematical Society, 1999.

\bibitem{vogel}
{\sc T. Vogel}, Existence of Engel structures,
{\it Ann.~of~Math.} (2){\bf 169} (2009), 79--137.

\bibitem{Wells}
{\sc R. O. Wells}, {\it Differential Analysis on Complex Manifolds}, 
Graduate Texts in Mathematics {\bf 65}, Springer-Verlag, 1980.
\end{thebibliography}
\end{document}